\documentclass[12pt,a4paper,reqno]{amsart}
\usepackage {amsfonts}
\usepackage{graphicx}

\newtheorem{theorem}{Theorem}
\newtheorem{lemma}[theorem]{Lemma}
\theoremstyle{remark}
\newtheorem{remark}[theorem]{Remark}
\newtheorem{condition}[theorem]{Condition}

\DeclareMathOperator{\minor}{minor}

\newcommand{\menshe}{\scriptstyle}

\newcommand{\promezhutok}{\;}

\author{I.G. Korepanov}
\thanks{South Ural State University, Chelyabinsk, Russia. E-mail: kig@susu.ac.ru}
\title[Geometric torsions and manifolds with boundary]{Geometric torsions and invariants of manifolds with triangulated boundary}
\date{}

\begin{document}

\begin{abstract}
Geometric torsions are torsions of acyclic complexes of vector spaces which consist of differentials of geometric quantities assigned to the elements of a manifold triangulation. We use geometric torsions to construct invariants for a manifold with a triangulated boundary. These invariants can be naturally united in a vector, and a change of the boundary triangulation corresponds to a linear transformation of this vector. Moreover, when two manifolds are glued by their common boundary, these vectors undergo scalar multiplication, i.e., they work according to M.~Atiyah's axioms for a topological quantum field theory.
\end{abstract}

\maketitle

\section{Introduction}

A topological quantum field theory in dimension~$d$, according to the axioms of M.~Atiyah, should bring in correspondence to a compact oriented $(d+1)$-dimensional manifold with an $m$-component boundary an ``$m$-legged object'', i.e., an element of the tensor product of $m$ finite-dimensional vector spaces over the field~$\mathbb C$, each of these corresponding to a component of the boundary. When we change the orientation of a boundary component, its vector space is replaced with its dual (the space of linear functionals on the initial space). The glueing of manifolds $M_1$ and~$M_2$ by some components of their boundaries is represented by glueing the corresponding ``legs'', and the mentioned objects are muliplied using the natural pairing of vector spaces, since each boundary component participating in the glueing has opposite orientations in $M_1$ and~$M_2$. A single number --- an element of~$\mathbb C$ --- corresponds to a manifold without boundary. For a full presentation of M.~Atiyah axioms, see his paper~\cite{atiyah} or book~\cite{atiyah1}.

In the present paper, we propose a way to construcing such a field theory using ``geometric torsions'', i.e., torsions of acyclic complexes of vector spaces consisting of differentials of geometric quantities assigned to the elements of a manifold triangulation. To be exact, we take a three-dimensional triangulated manifold, with or without boundary, and provide every simplex with a Euclidean geometry (in particular, edges acquire Euclidean lengths). This geometry implies by no means any Euclidean structure for the whole manifold; it provides algebraic quantities enabling us to build acyclic complexes from which we can obtain manifold invariants. Such invariant of geometric origin appeared for the first time in paper~\cite{3dcase}, while in papers~\cite{24,15} it was understood that the algebra behind it is the algebra of acyclic complexes.

Different versions of such invariants have been proposed. In particular, invariants obtained from acyclic complexes, ``twisted'' by a representation of the manifold's fundamental group, are of interest. It became clear during our work on them (although not yet proved in full generality) that they are related to both usual and non-abelian Reidemeister torsion and twisted Alexander polynomial. The first result on this subject is presented in paper~\cite{KM} (without using the language of acyclic complexes); further results are due to E.~Martyushev~\cite{M1,M2} and can be found in the most detailed form in his Ph.D.\ thesis~\cite{M-diss}. 

The feeling was, however, that the potential of geometric torsions was far richer and, besides, a desire remained to construct invariants not using representations of a fundamental group (which are very often hard to describe). This brought us to studying \emph{relative} invariants. An investigation of such invariant for a lens space without a tubular neighborhood of an unknot (i.e., the least knotted loop going along a given element of the fundamental group), undertaken in paper~\cite{dkm}, led us at once to nontrivial results and suggested an idea to construct in such way a topological field theory in the spirit of Atiyah's axioms, or even obeying these axioms literally. The most remarkable thing is that our geometric torsions exist, apparently, for \emph{any-dimen\-sion\-al} manifolds, see, for instance, papers~\cite{33,24,15} for the four-dimen\-sion\-al case.

In the present paper we, nevertheless, restrict ourselves to three-dimen\-sion\-al manifolds and, besides, having only one boundary component. The aim of the paper is to show how to construct a vector of invariants by means of geometric torsions and how a pairing of two such vectors is performed when two manifolds are glued together (the result of which is a manifold without boundary). The contents of the remaining sections of the paper is as follows. In section~\ref{closed}, we explain how to construct an algebraic complex and an invariant for a \emph{closed} manifold. This is exactly the invariant to be obtained as the pairing of invariant vectors for two manifolds with boundary. The construction of these vectors is explained in section~\ref{skraem}. Then, in section~\ref{izmenenie} we show that the invariant vector undergoes a linear transformation under a transition to a new boundary triangulation, and in section~\ref{skleivanie} we define a scalar product of two such vectors (for two manifolds with identically triangulated boundaries) and show that it yields the invariant of the result of the glueing. Some examples confirming the nontriviality of our invariants are presented in section~\ref{primery}. In the concluding section~\ref{discussion}, we discuss the results of the paper and prospects for further research.

\section{The algebraic complex and the simplest invariant for a closed three-dimensional manifold}
\label{closed}

We recall how we construct the simplest version of an algebraic complex for a triangulated closed oriented three-dimensional manifold~$M$ and the corresponding invariant (which is, essentially, the invariant from paper~\cite{3dcase}). The triangulation is not obliged to be strictly combinatorial, that is, we allow both multiple appearance of a simplex in the boundary of a simplex of a greater dimension and presence of several simplices with the same vertices. The following condition must, however, hold: all four vertices of any tetrahedron are different.

\subsection{Construction of the complex}

To each vertex~$A$ of the triangulation, we assign three real numbers $x_A,y_A,z_A$. These are parameters entering in our further constructions; we call them unperturbed Euclidean coordinates of a given vertex. From these coordinates, (unperturbed) Euclidean volumes are determined which we ascribe to simplices of various dimensions, for instance, the length of an edge~$AB$ is
\begin{equation}
l_{AB}=\sqrt{(x_B-x_A)^2+(y_B-y_A)^2+(z_B-z_A)^2}.
\label{dlina}
\end{equation}
The arrangement of vertices in a Euclidean three-dimensional space~$\mathbb R^3$ has no relation to the topology of~$M$, we only require that the coordinates should lie in the general position with respect to all further algebraic constructions. In particular, the oriented volume of any tetrahedron~$ABCD$
\begin{equation}
V_{ABCD}=\frac{1}{6}\left|
\begin{matrix}\overrightarrow{AB}& \overrightarrow{AC}& \overrightarrow{AD}\end{matrix}
\right|,
\label{vabcd}
\end{equation}
where the determinant is composed of three column vectors, does not vanish.

We now fix an orientation of $M$, that is, a consistent orientation of all tetrahedra; the orientation of a tetrahedron is a given order of its vertices to within their even permutations. In the sequel, while saying ``tetrahedron~$ABCD$'', we mean that this exactly ordering of its vertices yields its positive orientation. Nevertheless, the volume~(\ref{vabcd}) of tetrahedron~$ABCD$ can be of any sign, and we will ascribe the same sign to its \emph{inner dihedral angles} which we will need soon.

The algebraic complex consists of vector spaces of column vectors and their linear mappings, that is, matrices:
\begin{equation}
0\longrightarrow \mathfrak e(3) \stackrel{f_1}{\longrightarrow} (dx) \stackrel{f_2}{\longrightarrow} (dl) \stackrel{f_3=f_3^{\rm T}}{\longrightarrow} (d\omega) \stackrel{f_4=-f_2^{\rm T}}{\longrightarrow} (dx^*) \stackrel{f_5=f_1^{\rm T}}{\longrightarrow} \mathfrak e(3)^* \longrightarrow 0.
\label{complex-closed}
\end{equation}
Here $\mathfrak e(3)$ is the Lie algebra of infinitesimal motions of the three-dimensional Euclidean space, whose element we write as a column vector of height~6 in the natural basis of three translations and three rotations; $(dx)$ and $(dx^*)$ are vector spaces of columns of height~$3N_0$, where $N_0$ is the number of vertices in the triangulation; $(dl)$ and $(d\omega)$ are vector spaces of columns of height~$N_1$, where $N_1$ is the number of edges in the triangulation. We write a column vector from space~$(dx)$ as $(dx_A,dy_A,dz_A,\dots,dx_Z,dy_Z,dz_Z)^{\rm T}$, where $A,\dots, Z$ are triangulation vertices; mapping~$f_1$ gives by definition the infinitesimal translations of these vertices from their initial positions $(x_A,y_A,z_A),\dots,(x_Z,y_Z,z_Z)$ under the action of an element of the Lie algebra. We write a column vector from space~$(dl)$ as $(dl_1,\dots,dl_{N_1})^{\rm T}$, where the subscripts number the edges in the triangulation; mapping~$f_2$ gives by definition the infinitesimal changes of edge lengths corresponding to the given changes of vertex coordinates.

To describe mapping~$f_3$, we need the notion of \emph{deficit angle}~$\omega_i$ around a given edge~$i$. If all the edge lengths are obtained from some vertex coordinates according to~(\ref{dlina}), then the algebraic sum of dihedral angles at edge~$i$ (remember the sign with which we agreed to take such an angle!) $\sum_k \varphi_i^{(k)}=0 \bmod 2\pi$, where $k$ numbers the tetrahedra situated around edge~$i$. If, however, we change the edge lengths arbitrarily and independently, the deficit angle appears
\begin{equation}
\omega_i \stackrel{\rm def}{=} - \sum_k \varphi_i^{(k)} \bmod 2\pi.
\label{omegai}
\end{equation}
Mapping~$f_3$ gives the vector of infinitesimal deficit angles $(d\omega_1,\dots,d\omega_{N_1})^{\rm T}$ appearing from given infinitesimal deformations of lengths.

An important property of matrix~$f_3$, which facilitates considerably further construction of complex~(\ref{complex-closed}), is its \emph{symmetry} proved in paper~\cite{3dcase}.\footnote{A more elegant way of proving the symmetry of~$f_3$ can be extracted from section~4 of paper~\cite{33}, where we deal with a similar matrix, but for the case of a \emph{four-dimensional} manifold.} Thus, there is no need to care about the geometric sense of spaces denoted as $(dx^*)$ and~$\mathfrak e(3)^*$; we can instead simply consider them as spaces of column vectors between which matrices $f_4$ and~$f_5$ act, defined by transposing their ``mirror image'' matrices, as indicated above the arrows in~(\ref{complex-closed}).\footnote{The minus sign in the definition of $f_4$ is a correction which I owe to E.V.~Martyushev. It will play its role in the definition of torsions (\ref{tau-closed}) and~(\ref{tau-b}) and invariants obtained from them.}

\begin{theorem}
\label{th-complex-closed}
Sequence (\ref{complex-closed}) is an algebraic complex, i.e., the composition of any two successive arrows is zero.
\end{theorem}

\begin{proof}
The equalities $f_2\circ f_1=0$ and $f_3\circ f_2=0$ follow from simple geometric reasoning: the former from the fact that motions of a Euclidean space as a whole do not change edge lengths, and the latter from the fact that any changes of edge lengths which are due only to some changes of vertex coordinates leave deficit angles zero. Transposing these equalities, we get $f_5\circ f_4=0$ and $f_4\circ f_3=0$.
\end{proof}

\begin{remark}
At the same time, complex (\ref{complex-closed}) may not be acyclic. According to the results of section~2 of paper~\cite{24}, the ``twisted'' complex is acyclic, one that involves the universal covering of a triangulated manifold~$M$, a representation of the fundamental group of~$M$ into the group of motion of the Euclidean space $\rho\colon\; \pi_1(M)\to \mathrm{E}(3)$, and infinitesimal deformations of representation~$\rho$. Still, there are cases where complex~(\ref{complex-closed}) is acyclic (if the group~$\pi_1(M)$ is finite and, accordingly, $\rho$ cannot be deformed into a non-equivalent one), so, for the time being, we restrict ourselves to it in order not to overload our exposition and because we will use this complex exactly below in theorem~\ref{thglueing}.
\end{remark}

\begin{remark}
\label{remark-acycl-2}
More specifically, the acyclicity can be violated in terms $(dl)$ and~$(d\omega)$. On the other hand, any changes of coordinates preserving all lengths can be obtained from a motion of the Euclidean space, which leads to acyclicity in term~$(dx)$ and, by symmetry, in~$(dx^*)$. As evident is the injectivity of mapping~$f_1$, which leads to acyclicity in term~$\mathfrak e(3)$ and, by symmetry, in~$\mathfrak e(3)^*$.
\end{remark}

\subsection{Torsion and the invariant}

We define the torsion of complex~(\ref{complex-closed}) by formula
\begin{equation}
\tau \stackrel{\rm def}{=} \frac{\minor f_1 \, \minor f_3 \, \minor f_5}{\minor f_2 \, \minor f_4},
\label{tau-closed}
\end{equation}
where the minors of matrices are chosen according to the standard definition of a nondegenerate $\tau$-chain~\cite{T}, if it exists. We are going to give more details on this, as well as specify our choice of minors in the case if there is no nondegenerate $\tau$-chain.

To choose a $\tau$-chain means to choose some basis vectors from the fixed basis of each vector space in an algebraic complex.\footnote{The fact that we regard, from the very beginning, all vector spaces as spaces of column vectors means exactly, of course, that we have chosen fixed bases in them.} These basis vectors correspond to the rows which we choose for a next coming minor, while the rest of basis vectors correspond to the columns of the preceding minor, and so on. A $\tau$-chain is called nondegenerate if all its minors are nonzero; it always exists for an acyclic complex. We would like, however, to use formula~(\ref{tau-closed}) even in the absence of acyclicity. It is not hard to deduce from remark~\ref{remark-acycl-2} that there always exists a $\tau$-chain for which all minors in formula~(\ref{tau-closed}) are nonzero, \emph{except} for, maybe, $\minor f_3$. This is the way how we will always choose our $\tau$-chain.

Torsion~(\ref{tau-closed}) is inverse to the torsion defined according to paper~\cite{24} and the rest of our preceding works; the same applies to invariant~$I$ introduced below. The reason for switching to the new definition is that we will deal below in theorems \ref{thlinear} and~\ref{thglueing} with \emph{linear combinations} of torsions of the kind~(\ref{tau-closed}), to be exact, of its analogues~(\ref{tau-b}) for manifolds with boundary. We note also that for a non-acyclic complex (\ref{complex-closed}) (or complex~(\ref{complex-b}) below), we get $\tau=0$ with our definition~(\ref{tau-closed}), while there would be an infinite torsion with our old definition.

An easy adaptation of theorem~1 from paper~\cite{24} (and a small correction concerning the right placing of minus signs) gives the following theorem.

\begin{theorem}
\label{th-inv-closed}
The value
\begin{equation}
I(M)=\frac{\tau \prod_{\textrm{over all tetrahedra}}(-6V)}{\prod_{\textrm{over all edges}}l^2}
\label{inv-closed}
\end{equation}
is an invariant of manifold~$M$. \qed
\end{theorem}

\begin{remark}
The symmetry of complex~(\ref{complex-closed}) removes the \emph{sign} problem for torsion and, accordingly, for invariant~$I(M)$: the torsion is determined unambiguously if we take a diagonal minor of~$f_3$ (the rows with the same numbers as the columns), and the minors of $f_4$ and~$f_5$ symmetric to minors of $f_2$ and~$f_1$, respectively. We imply in theorem~\ref{th-inv-closed} that we have chosen the minors in this way.
\end{remark}

\section{The algebraic complex and the invariant vector for a manifold with a triangulated boundary}
\label{skraem}

Let now~$M$ be a three-dimensional compact orientable manifold with boundary, and let there be given a fixed triangulation of the boundary, while it is permitted to change the triangulation within the manifold. When constructing an analogue of complex~(\ref{complex-closed}), we must take into account that vertices and edges are divided into inner ones and those lying in the boundary (an edge is considered boundary only it lies entirely within the boundary). No deficit angle~(\ref{omegai}) is defined for boundary edges, so we take instead the most similar quantity
\begin{equation}
\alpha_i \stackrel{\rm def}{=} - \sum_k \varphi_i^{(k)} \bmod 2\pi,
\label{alphai}
\end{equation}
that is, minus dihedral angle at the given edge. Besides, wishing to get as many invariants as possible, we choose and fix for a while two arbitrary subsets $\mathcal C$ and~$\mathcal D$ \emph{of equal cardinality} of the set of boundary edges.

We define the following contracted modification of complex~(\ref{complex-closed}):
\begin{equation}
0\longrightarrow (dx_{\textrm{inner}}) \stackrel{f_2}{\longrightarrow} (dl_{\textrm{inner}},dl_{\textrm{boundary}}) \stackrel{f_3}{\longrightarrow} (d\omega,d\alpha) \stackrel{f_4=-f_2^{\rm T}}{\longrightarrow} (dx^*_{\textrm{inner}}) \longrightarrow 0.
\label{complex-b}
\end{equation}
Thus, only coordinate differentials for \emph{inner} vertices participate now in our complex: a column vector from space~$(dx_{\textrm{inner}})$ consists by definition of these differentials and only of them. Further, by definition, a column vector from space $(dl_{\textrm{inner}},\allowbreak dl_{\textrm{boundary}})$ consists of length differentials for \emph{all inner} edges, as well as those boundary edges entering the set~$\mathcal C$. A column vector from space~$(d\omega,d\alpha)$ consists by definition of deficit angle differentials for all inner edges and $d\alpha_i$ for edges~$i\in \mathcal D$. Finally, in the space~$(dx^*_{\textrm{inner}})$ as well, we include the quantities belonging to all inner vertices and only to them.\footnote{We defined space $(dx^*)$ in section~\ref{closed} as the space of column vectors without going into their geometric sense. Nevertheless, every column element obviously belongs to a certain vertex.}

\begin{theorem}
Sequence (\ref{complex-b}) is an algebraic complex for any choice of sets $\mathcal C$ and~$\mathcal D$.
\end{theorem}

\begin{proof}
The equality $f_3\circ f_2$ follows from the same geometric considerations as in the proof of theorem~\ref{th-complex-closed}. One should just notice that not only deficit angles remain zero when we move inner vertices but also dihedral angles~$\alpha$ at boundary edges do not change as long as the boundary vertices do not move.

The equality $f_4\circ f_3$ follows from the preceding if we interchange $\mathcal C \leftrightarrow \mathcal D$ and transpose the matrices.
\end{proof}

If complex~(\ref{complex-b}) is acyclic, we define its torsion by the formula
\begin{equation}
\tau \stackrel{\rm def}{=} \frac{\minor f_3}{\minor f_2 \, \minor f_4},
\label{tau-b}
\end{equation}
where the minors are chosen according to the rule for a nondegenerate $\tau$-chain and, besides, minors of $f_2$ and~$f_4$ must be mutually symmetric. Even if complex~(\ref{complex-b}) is not acyclic, the rectangular matrix~$f_2$ always has a nonzero minor of the maximal size (equal to the tripled number of inner vertices): this can be easily understood if one counts the inner vertices in a certain order, every time putting three edges in correspondence to the vertex. Namely, we start from a vertex joined (at least) by three edges to the \emph{boundary}, and include in the minor of~$f_2$ three columns corresponding to this vertex and three rows corresponding to these edges; then we choose at every step a vertex joined to the boundary and/or to already chosen vertices by three edges and proceed in a similar way. The obtained minor has a block triangular form (with $3\times 3$ blocks on the diagonal), and it makes no difficulty to check that it does not equal zero (recall that all vertex coordinates are in the general position!). We will assume thus that such exactly minor of~$f_2$ and its symmetric minor of~$f_4$, \emph{the same for all $\mathcal C$ and~$\mathcal D$}, are chosen in formula~(\ref{tau-b}), while we take for the minor of~$f_3$ the rows and columns corresponding to the edges not involved in the minors of $f_4$ and~$f_2$ respectively. This extendes definition~(\ref{tau-b}) to the non-acyclic case $\minor f_3=0$ as well.

\begin{theorem}
\label{th-b}
The value
\begin{equation}
I_{\mathcal C,\mathcal D}(M)=\frac{\tau \prod_{\textrm{over all tetrahedra}}(-6V)}{\prod_{\textrm{over inner edges}}l^2}
\label{inv-b}
\end{equation}
for given sets of boundary edges $\mathcal C$ and~$\mathcal D$ is an invariant of manifold~$M$ with triangulated boundary; it is a function of boundary vertex coordinates.
\end{theorem}

We call the cardinality of any of the sets $\mathcal C$ and~$\mathcal D$ \emph{level} of the corresponding invariant. For instance, a \emph{zero level invariant} corresponds to empty sets; such were the invariants studied in paper~\cite{dkm}.

\begin{proof}[Proof of theorem \ref{th-b}]
In paper~\cite{dkm}, the proof of a similar statement is presented for the particular case of zero level invariants and a torus  triangulated in a special way as the manifold boundary. Recall that it consists in two main parts: first, we prove that one can pass from a triangulation to another one by ``relative'', that is, not touching the boundary, Pachner moves $2\leftrightarrow 3$ and~$1\leftrightarrow 4$, and second, we investigate what happens to matrices entering the complex under these moves. Regardless of the fact that we are now in a more general situation, the reasoning from~\cite{dkm} is still fully applicable to our case, so we refer the reader to the mentioned work for details.
\end{proof}

\begin{remark} \label{remark-sign-ICD}
The question with the \emph{signs} of invariants $I_{\mathcal C,\mathcal D}(M)$ is a bit more complicated than for closed manifolds: the minors of $f_5$ and~$f_4$ are still symmetric to the minors of $f_1$ and~$f_2$ respectively, but the sign of the minor of~$f_3$, with $\mathcal C \ne \mathcal D$, does depend on the edge ordering. One can see, however, that it is enough to fix an ordering of only the \emph{boundary} edges, so this dependence does not bring about any big difficulties.
\end{remark}

Certainly, all possible invariants for various $\mathcal C, \mathcal D$ are far from being independent; we discuss it in detail in section~\ref{discussion}.

\section{The transformation of invariant vector under a change of boundary triangulation}
\label{izmenenie}

We now consider what happens to the set of invariants (\ref{inv-b}) taken for all possible $\mathcal C, \mathcal D$ under a change of the triangulation of manifold~$M$'s boundary~$\partial M$. To pass from a triangulation of~$\partial M$ to another, one can use a sequence of two-dimensional Pachner moves $2\to 2$ and~$1\leftrightarrow 3$.

Move~$2\to 2$ corresponds to glueing a new tetrahedron to the boundary by its two faces, say, tetrahedron~$ABCD$ by faces $ABC$ and~$BDC$. Edge~$BC$ becomes inner, instead, new edge~$AD$ appears on the boundary. We choose some edge sets $\mathcal C$ and~$\mathcal D$ for the situation \emph{after} the move~$2\to 2$ and try to express the minors entering in formula~(\ref{tau-b}) in terms of those given \emph{before} this move.

Firstly, we retain \emph{the same} minors of matrices $f_2$ and~$f_4$, chosen as described after formula~(\ref{tau-b}), because the rows and columns needed for them have not undergone any changes (recall that, moreover, these minors do not depend on $\mathcal C$ and~$\mathcal D$). Now we study what happens to matrix~$f_3$.

Consider first the ``full'' matrix~$f_3$, corresponding to the case $\mathcal C=\mathcal D=\{$all boundary edges$\}$. The changes in it can be described as follows: expand~$f_3$ first with a zero row and a zero column which will correspond to edge~$AD$. Then add to the obtained matrix the matrix $(\partial \alpha_i / \partial l_j)$ of partial derivatives of minus dihedral angles with respect to edge lengths in the new tetrahedron~$ABCD$, also expanded with necessary zero rows and columns.

As for the matrix of minor~$f_3$ for arbitrary $\mathcal C, \mathcal D$, it, being a submatrix of the ``full'' matrix, is obtained from some submatrix of the ``full'' matrix, taken before the move~$2\to 2$, in a similar way: expanding, if necessary, by a new row and column and adding the relevant submatrix of~$(\partial \alpha_i / \partial l_j)$. The key for calculating the new minor is given by the following lemma.

\begin{lemma}
\label{lemmaA+B}
The determinant of a sum of matrices $A+B$ is
\begin{equation}
\det(A+B)=\sum \epsilon \, \minor A \, \minor B,
\label{A+B}
\end{equation}
where the summation goes over \emph{all} minors of matrix~$A$, starting with zero size and through the full size; for the minor of~$B$, complementary rows and columns to minor~$A$ are taken; the sign $\epsilon=\pm 1$ corresponds to the sign at the cofactor of minor of~$A$.
\end{lemma}

\begin{proof}
Write $\det(A+B)$ as an alternated sum of products of matrix~$A+B$ elements and expand the parentheses in every term. Clearly, the obtained alternated sum consists exactly of the same summands as the right-hand side of~(\ref{A+B}).
\end{proof}

Assuming that the matrix of the new minor of~$f_3$ plays the role of~$(A+B)$ in lemma~\ref{lemmaA+B}, while the relevant submatrix of $(\partial \alpha_i / \partial l_j)$ for tetrahedron~$ABCD$ --- the role of~$A$, we get a linear combination of quantities $\minor B$, and it is not hard to see that each of these latter minors, entering in the sum~(\ref{A+B}) with a nonvanishing coefficient, is exactly a minor of~$f_3$ from formula~(\ref{tau-b}) for some $\mathcal C, \mathcal D$ in the situation \emph{before} the move~$2\to 2\,$!

Considering the invariants~(\ref{inv-b}) for various $\mathcal C, \mathcal D$ as components of the \emph{invariant vector}~$\vec I(M)$, we get the following lemma.

\begin{lemma} \label{lemma22}
To a Pachner move $2\to 2$ on the boundary of manifold~$M$ corresponds a linear transformation of the invariant vector~$\vec I(M)$.  
\end{lemma}

\begin{proof}
It remains to note that a new multiplier of the form $(-6V)/l^2$ in formula~(\ref{inv-b}), caused by the appearance of a new tetrahedron and a new inner edge, does not violate the linearity of transformation of the set of minors~$f_3$ and, hence, of vector~$\vec I(M)$.
\end{proof}

A similar result holds also for moves~$1\leftrightarrow 3$. Move $1\to 3$ differs from $2\to 2$, from this viewpoint, only in the way of glueing a new tetrahedron: it is glued by \emph{one} face. Consider a bit more complicated case~$3\to 1$. Let there have been a triangle~$ABC$ in the boundary triangulation, divided in three triangles $ABD$, $BCD$ and~$CAD$. To the move $3\to 1$ corresponds the glueing of a tetrahedron~$DABC$ by these three faces. After doing this, vertex~$D$ and edges $AD$, $BD$ and~$CD$ become \emph{inner}. Thus, matrix~$f_2$ acquires three new columns, corresponding to partial derivatives with respect to $x_D$, $y_D$ and~$z_D$. Accordingly, $\minor f_2$ in formula~(\ref{tau-b}) must as well acquire these three columns, and also three rows corresponding to lengths $l_{AD}$, $l_{BD}$ and~$l_{CD}$. Standard reasoning based on block triangularity shows that, as a result, $\minor f_2$ is simply multiplied by the determinant of Jacobian matrix of partial derivatives of the three mentioned lengths with respect to three coordinates of~$D$; the same applies to the minor of~$f_4$. Such multiplication does not violate the linearity, so we get one more lemma.

\begin{lemma} \label{lemma13}
To Pachner moves $1\leftrightarrow 3$ on the boundary of manifold~$M$ also correspond linear transformations of the invariant vector~$\vec I(M)$.
\qed
\end{lemma}

The final result of lemmas \ref{lemma22} and~\ref{lemma13} is the following theorem.

\begin{theorem} \label{thlinear}
A linear transformation~$\mathcal A$ can be associated with a transition from one triangulation of the boundary of a compact three-dimensional manifold~$M$ to another its triangulation, in such way that the invariant vector~$\vec I(M)$ goes into $\mathcal A \vec I(M)$, and $\mathcal A$ can be chosen to depend only on the initial and final boundary triangulations, but not on~$M$ itself.
\end{theorem}

\begin{proof}
Consider a chain of Pachner moves transforming the old boundary triangulation into the new one. According to the preceding reasoning, the glueing of the corresponding tetrahedron sequence determines a chain of linear transforms acting on~$\vec I(M)$.
\end{proof}

\begin{remark}
We have not proved that two different chains of Pachner moves, yielding the same transformation of the boundary triangulation, will give in this way the same~$\mathcal A$. Even if it were not so, new vectors $\mathcal A \vec I(M)$, for any fixed~$M$ and 	hypothetic different~$\mathcal A$ must coincide. Our conjecture is that the algorithm of constructing~$\mathcal A$ presented above leads really to a result not depending on the choice of a specific chain.
\end{remark}

\section{Glueing of manifolds and a composition of invariants}
\label{skleivanie}

Let $M_1$ and~$M_2$ be compact oriented three-dimensional manifolds with identical but oppositely oriented one-component boundaries: $\partial M_1=-\partial M_2$. Let also identical triangulations be given on $\partial M_1$ and~$\partial M_2$, so that $M_1$ and~$M_2$ can be glued together by identifying the corresponding simplices on their boundaries. The result of this glueing will be a closed manifold~$M$.

We choose for $M_1$ and~$M_2$ some triangulations whose restriction onto the boundary coincides with its given triangulation. In this way $M$ as well gets triangulated; we consider the algebraic complex~(\ref{complex-closed}) written for it. We are going to calculate the torsion of this complex, choosing a $\tau$-chain, i.e., minors, in such way that it can be related to the torsions of complexes~(\ref{complex-b}) written for $M_1$ and~$M_2$.

We will need the following technical condition on the triangulation of surface~$\Gamma \subset M$ --- the common boundary of $M_1$ and~$M_2$ considered regardless of its orientation.

\begin{condition} \label{uslovie}
The vertices in surface $\Gamma$ can be written in such order $A,B,\allowbreak C,\allowbreak D,\dots,Z$ that
\begin{itemize}
\item[a)] vertices $A$, $B$ and~$C$ form a triangle, i.e., they are joined to each other by edges,
\item[b)] each vertex, starting from~$D$, is joined by edges to at least three preceding vertices.
\end{itemize}
All edges are supposed to lie in~$\Gamma$.
\end{condition}

A triangulation of~$\Gamma$ satisfying condition~\ref{uslovie} certainly exists, it can be easily constructed for a surface~$\Gamma$ of any genus~$g$. Represent~$\Gamma$ in a standard way as a $4g$-gon with the following sides, listed in the order as we go around it in the positive direction: $a_1,\allowbreak b_1,\allowbreak a_1^{-1},\allowbreak b_1^{-1}\dots,\allowbreak a_g,\allowbreak b_g,\allowbreak a_g^{-1},\allowbreak b_g^{-1}$. All vertices of this polygon, as is known, are identified to just one vertex; denote it~$D$. We place one more vertex, $C$, in the center of the polygon, and $2g$ more vertices will be situated in the middles of the polygon sides in the following order: $A,\allowbreak B,\allowbreak A,\allowbreak B,\allowbreak E,\allowbreak F,\allowbreak E,\allowbreak F,\dots$. Finally, we join the centers of successive sides with one another, as well as with the polygon center, see figure~\ref{risunok1}.
\begin{figure}
\includegraphics[scale=0.24]{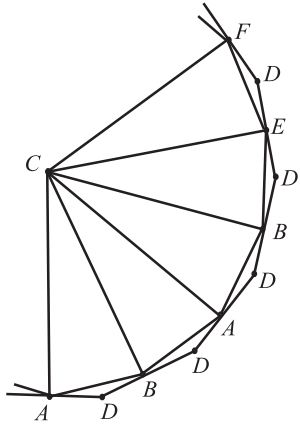}
\caption{Standard triangulation of the common boundary~$\Gamma$ of manifolds $M_1$ and~$M_2$}
\label{risunok1}
\end{figure}

Below we assume that a triangulation has been chosen for~$\Gamma$ for which condition~\ref{uslovie} does hold, for instance, one depicted in figure~\ref{risunok1}.

We order and denote vertices in $\Gamma$ as in condition~\ref{uslovie}. We take for the minor of matrix~$f_1$ its rows corresponding to coordinate differentials $dx_A,\allowbreak dy_A,\allowbreak dz_A,\allowbreak dy_B,\allowbreak dz_B,\allowbreak dz_C$. With~$f_5$, we deal in the symmetrical way, i.e., we take for the minor of this matrix the \emph{columns} corresponding to the same differentials.

Now we turn to the minor of~$f_2$ whose rows correspond to edges in the triangulation or, to be exact, to their length differentials. As for edges in~$\Gamma$, we include in the minor of~$f_2$ the rows corresponding to $dl_{AB},dl_{BC},dl_{CA}$, and also three ones for any of the further vertices $D,\dots,Z$, corresponding to the edges mentioned in condition~\ref{uslovie}. Besides, we include in the minor of~$f_2$ for~$M$ all the rows which were present in the corresponding minors with which we calculated torsions~(\ref{tau-b}) for $M_1$ and~$M_2$. Recall that the way of constructing these minors is regarded as fixed and is explained after formula~(\ref{tau-b}). Then we, of course, choose the minor of~$f_4$ symmetrical to the minor of~$f_2$.

We keep the notation $\minor f_2$ for the minor belonging to the complex~(\ref{complex-closed}) written for~$M$, while we introduce notations $(\minor f_2)^{(1)}$ and~$(\minor f_2)^{(2)}$ for the minors belonging to complexes~(\ref{complex-b}) written for~$M_1$ and~$M_2$ respectively. Then
\begin{equation}
\minor f_2 = (\minor f_2)^{(1)} (\minor f_2)^{(2)} (\minor f_2)^{\Gamma},
\label{f2skleivanie}
\end{equation}
where
\begin{equation}
(\minor f_2)^{\Gamma} = \frac{dl_{AB}\wedge dl_{BC}\wedge dl_{CA}}{dx_B\wedge dx_C\wedge dy_C} \cdot \frac{dl_{AD}\wedge dl_{BD}\wedge dl_{CD}}{dx_D\wedge dy_D\wedge dz_D} \cdot \ldots
\label{f2Gamma}
\end{equation}
is the factor belonging to the boundary in the minor of~$f_2$. We took the liberty to write the determinants of Jacobian matrices (of size $3\times 3$) as ratios of exterior products of differentials. Of course, equalities (\ref{f2skleivanie}) and~(\ref{f2Gamma}) are obtained using the usual block triangularity of matrices, and the dots in the end of~(\ref{f2Gamma}) imply that the next factor has in its denominator the differentials of three coordinates of the next vertex~$E$, and in its numerator --- the length differentials for the three edges which join $E$ to preceding vertices according to condition~\ref{uslovie}, and so on.

Similarly, we keep the notation $\minor f_3$ for the minor belonging to the complex~(\ref{complex-closed}) constructed for~$M$, while we introduce notations $(\minor f_3)_{\mathcal C_1, \mathcal D_1}^{(1)}$ and~$(\minor f_3)_{\mathcal C_2, \mathcal D_2}^{(2)}$ for the minors belonging to the complexes~(\ref{complex-b}) written for~$M_1$ and~$M_2$ respectively, where $\mathcal C_i, \mathcal D_i$ are sets $\mathcal C$ and~$\mathcal D$ of edges (introduced in the beginning of section~\ref{skraem}), its own for each of $M_i,\,i=1,2$. To calculate~$\minor f_3$, we use again lemma~\ref{lemmaA+B}. We get:
\begin{equation}
\minor f_3 = \sum_{\mathcal C_1, \mathcal D_1} \epsilon \, (\minor f_3)_{\mathcal C_1, \mathcal D_1}^{(1)} (\minor f_3)_{\overline{\mathcal C}_1, \overline{\mathcal D}_1}^{(2)},
\label{f3skleivanie}
\end{equation}
where $\overline{\mathcal C}_1$ and~$\overline{\mathcal D}_1$ are complements of sets $\mathcal C_1$ and~$\mathcal D_1$ respectively in the set~$\mathcal E$ of all edges in~$\Gamma$ \emph{minus those edges dealt with in condition~\ref{uslovie}} and which are already involved in the minor of~$f_2$ and its symmetrical minor of~$f_4$. Thus, $\mathcal E$ consists of edges in~$\Gamma$, except $AB,\allowbreak BC,\allowbreak CA,\allowbreak AD,\allowbreak BD,\allowbreak CD$ and then three more edges for every new vertex.

\begin{theorem} \label{thglueing}
The invariant $I(M)$ of the closed manifold~$M$ obtained by glueing two manifolds $M_1$ and~$M_2$ with one-component boundary is equal to the following scalar product of invariant vectors $\vec I(M_1)$ and~$\vec I(M_2)$:
\begin{equation}
I(M) = c \, \sum_{\mathcal C_1, \mathcal D_1} \epsilon \, I_{\mathcal C_1, \mathcal D_1}(M_1) \, I_{\overline{\mathcal C}_1, \overline{\mathcal D}_1}(M_2),
\label{invglueing}
\end{equation}
where the quantity~$c$ belongs only to the boundary~$\Gamma$ by which the glueing goes:
\begin{multline}
c = \frac{(\minor f_1)^2}{ {\displaystyle(-1)^s} \, \bigl((\minor f_2)^{\Gamma}\bigr)^2 \, \prod_{\textrm{over edges in }\Gamma} l^2 } \\
= \frac{(-1)^s}{\prod_{\textrm{over edges in }\Gamma} l^2 } \left( l_{AB}l_{BC}l_{CA} \cdot \frac{l_{AD}l_{BD}l_{CD}}{6V_{ABCD}} \cdot \ldots \right)^2 \\
= \frac{(-1)^s }{\prod_{i\in \mathcal E} l_i^2 \cdot (6V_{ABCD} \cdot \ldots )^2},
\label{velichina_c}
\end{multline}
$s$ is the number of rows in~$(\minor f_2)^{\Gamma}$. Each of the factors denoted by dots in the second line of formula~(\ref{velichina_c}) is analogous to $\frac{l_{AD}l_{BD}l_{CD}}{6V_{ABCD}}$, but instead of~$D$, for a next coming factor we take the next coming vertex according to condition~\ref{uslovie}, while instead of $A$, $B$ and~$C$ --- three vertices with which it is joined according to the same condition.
\end{theorem}

\begin{proof}
Formula (\ref{invglueing}) together with the first equality in formula~(\ref{velichina_c}) is obtained by juxtaposing formulas (\ref{tau-closed}), (\ref{inv-closed}), (\ref{tau-b}), (\ref{inv-b}), (\ref{f2skleivanie}) and~(\ref{f3skleivanie}) and taking into account the equalness of the minors of $f_1$ and~$f_5$ and the equalness up to a sign~$(-1)^s$ of the minors of $f_2$ and~$f_4$. By the way, it is not hard to deduce from the fact that $s$ is the number of edges mentioned in condition~\ref{uslovie} that
$$
s = 3 \cdot (\hbox{number of vertices in }\Gamma) - 6.
$$
The second equality in~(\ref{velichina_c}) is obtained by a direct calculation of minors: for $\minor f_1$ --- using formulas showing how elements of~$\mathfrak e(3)$ --- three infinitesimal translations and three rotaions --- act on coordinates~$x_A,y_A,\allowbreak z_A,\allowbreak y_B,\allowbreak z_B,z_C$, and for $(\minor f_2)^{\Gamma}$ --- using formula~(\ref{f2Gamma}). The third equality in~(\ref{velichina_c}) follows from the definition of set~$\mathcal E$ after canceling the squared lengths. Certainly, volume $V_{ABCD}$ and the next ones, concealed behind the dots, have purely formal sense and may not correspond to any tetrahedra in the triangulation, since points $A,\allowbreak B,\allowbreak C,\allowbreak D,\dots$ lie in the \emph{surface}~$\Gamma$.
\end{proof}

\begin{remark}
The sign $\epsilon$ in formulas (\ref{f3skleivanie}) and~(\ref{invglueing}) is, certainly, also determined from information belonging only to~$\Gamma$. Namely, $\epsilon = (-1)^{\sigma}$, where $\sigma$ is the sum of numbers of all elements in $\mathcal C_1$ and~$\mathcal D_1$ for the given ordering of vertices in~$\Gamma$, cf.~with remark~\ref{remark-sign-ICD}.
\end{remark}

\section{Examples}
\label{primery}

\subsection{Filled torus}
\label{polnotorie}

Let $M$ be a filled torus and, accordingly, its boundary~$\partial M$ --- two-dimensional torus. We glue the filled torus out of six tetrahedra in the following way. First, we take two identical tetrahedra $ABCD$ and glue them together by their \emph{edges}: glue edge~$AD$ of one tetrahedron to edge~$AD$ of the other one, and do the same with edges~$BC$. A chain of two tetrahedra appears; what prevents it from being a filled torus is its ``zero thickness'' at the places of glueing. Imagine that this chain is arranged as in figure~\ref{2tetraedra}.
\begin{figure}
\includegraphics[scale=0.32]{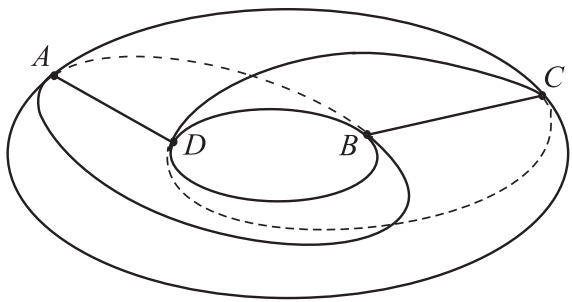}
\caption{The beginning of construction of a triangulated filled torus --- a chain of two tetrahedra~$ABCD$}
\label{2tetraedra}
\end{figure}
Now we glue at ``our'' side one more tetrahedron~$DABC$ (of the opposite orientation!) to faces $ADC$ and~$ABD$. This already creates a nonzero thickness at the edge~$AD$, yet we glue still one more tetrahedron~$ABCD$ to the two free faces of the new tetrahedron~$DABC$, that is, $ABC$ and~$BDC$. We do this having in mind to obtain the triangulation of~$\partial M$ depicted in figure~\ref{razvertka}. In the very same way, in order to remove the zero thickness along edge~$BC$, we glue one more tetrahedron of the opposite orientation at ``our'' side of the figure to faces $BAC$ and~$CDB$, and then glue a tetrahedron of the ``usual'' orientation to the two free faces of the new tetrahedron.

To distinguish edges of the same name, we introduce the following notations. Edges $AD$ and~$BC$ present in figure~\ref{2tetraedra} will be denoted $AD_0$ and~$BC_0$. Then, we think of one of the tetrahedra in figure~\ref{2tetraedra} as first, and the other as second, and accordingly assign to the rest of their edges indices $1$ or~$2$. It remains to denote four edges, of which two lie inside the filled torus (except their ends, of course), and two --- in the boundary. We denote the inner edges as $AD_3$ and~$BC_3$, and the boundary ones --- as $AD_4$ and~$BC_4$.

The obtained triangulated filled torus is depicted in figure~\ref{ris-polnotorie},
\begin{figure}
\includegraphics[scale=0.32]{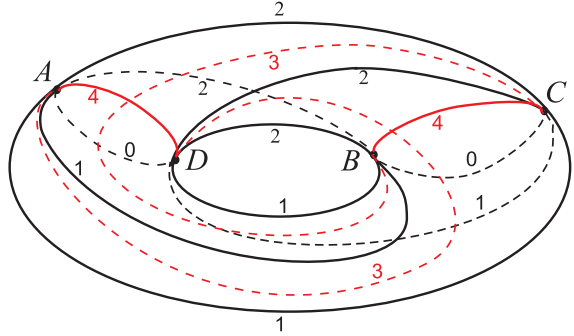}
\caption{Triangulated filled torus
(the numbers correspond to the subscripts at edges)%
}
\label{ris-polnotorie}
\end{figure}
and the development of the triangulation of torus --- its boundary --- in figure~\ref{razvertka}.
\begin{figure}
\includegraphics[scale=0.24]{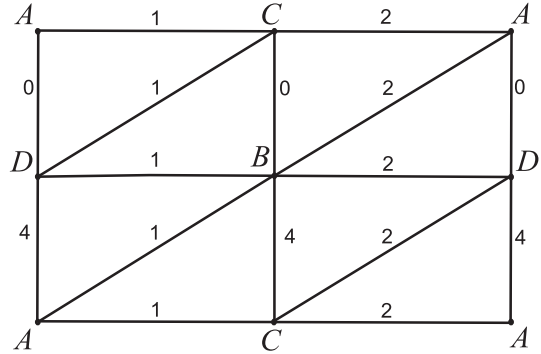}
\caption{Development of the triangulation of torus
(the numbers correspond to the subscripts at edges)%
}
\label{razvertka}
\end{figure}

The ``full'' matrix $f_3$ is of sizes $14\times 14$. Its matrix elements can be, nevertheless, represented in a transparent form, if we write out not themselves but their ratios to some standard expressions, namely, partial derivatives of minus dihedral angles~$\alpha_i$ with respect to lengths~$l_j$ in a tetrahedron~$ABCD$, where $i$ and~$j$ are edges with the same names as in matrix~$f_3$, but without indices. These ratios are shown in table~\ref{f3polnotorie}.\footnote{The method of calculating the entries in table~\ref{f3polnotorie} is easy: look which of the six tetrahedra in the triangulation are common for two given edges. Assign a number~$\pm 1$ to each tetrahedron according to its orientation, and sum up these numbers.}
\begin{table}
$$
\begin{array}{r|c@{\promezhutok} c@{\promezhutok} c@{\promezhutok} c@{\promezhutok} c@{\promezhutok} c@{\promezhutok} c@{\promezhutok} c@{\promezhutok} c@{\promezhutok} c@{\promezhutok} c@{\promezhutok} c@{\promezhutok} c@{\promezhutok} c}
& \menshe AC_1 &\menshe AB_1 &\menshe AD_4 &\menshe BD_1 &\menshe CD_1 &\menshe AD_0 &\menshe AC_2 &\menshe CD_2 &\menshe BC_4 &\menshe BD_2 &\menshe AB_2 &\menshe BC_0 &\menshe BC_3 &\menshe AD_3 \\ \hline
AC_1 & 1 & 1 & 0 & 1 & 1 & 1 & 0 & 0 & 1 & 0 & 0 & 0 & 0 & 0 \\
AB_1 & 1 & 1 & 1 & 1 & 1 & 0 & 0 & 0 & 1 & 0 & 0 & 0 & 0 & 0 \\
AD_4 & 0 & 1 & 1 & 1 & 0 & 0 & 1 & 1 & 0 & 0 & 0 & 0 & 1 & 0 \\
BD_1 & 1 & 1 & 1 & 1 & 1 & 0 & 0 & 0 & 0 & 0 & 0 & 1 & 0 & 0 \\
CD_1 & 1 & 1 & 0 & 1 & 1 & 1 & 0 & 0 & 0 & 0 & 0 & 1 & 0 & 0 \\
AD_0 & 1 & 0 & 0 & 0 & 1 & 1 & 0 & 0 & 0 & 1 & 1 & 2 &-1 & 0 \\
AC_2 & 0 & 0 & 1 & 0 & 0 & 0 & 1 & 1 & 0 & 1 & 1 & 1 & 0 & 0 \\
CD_2 & 0 & 0 & 1 & 0 & 0 & 0 & 1 & 1 & 1 & 1 & 1 & 0 & 0 & 0 \\
BC_4 & 1 & 1 & 0 & 0 & 0 & 0 & 0 & 1 & 1 & 1 & 0 & 0 & 0 & 1 \\
BD_2 & 0 & 0 & 0 & 0 & 0 & 1 & 1 & 1 & 1 & 1 & 1 & 0 & 0 & 0 \\
AB_2 & 0 & 0 & 0 & 0 & 0 & 1 & 1 & 1 & 0 & 1 & 1 & 1 & 0 & 0 \\
BC_0 & 0 & 0 & 0 & 1 & 1 & 2 & 1 & 0 & 0 & 0 & 1 & 1 & 0 &-1 \\
BC_3 & 0 & 0 & 1 & 0 & 0 &-1 & 0 & 0 & 0 & 0 & 0 & 0 & 0 & 0 \\
AD_3 & 0 & 0 & 0 & 0 & 0 & 0 & 0 & 0 & 1 & 0 & 0 &-1 & 0 & 0 \\
\end{array}
$$ 
\caption{Factors by which partial derivatives $\partial \alpha_i / \partial l_j$, taken in a single tetrahedron~$ABCD$, are multiplied to obtain matrix~$f_3$ elements for the filled torus}
\label{f3polnotorie}
\end{table}

The zero level invariant corresponds to the $2\times 2$ minor of the full matrix~$f_3$, obtained from the two last rows and two last columns in table~\ref{f3polnotorie}. As this minor consists entirely of zeros, not only this invariant, but also all level-one invariants vanish. Most level-two invariants vanish as well, and those non-vanishing are characterized by the fact that both sets $\mathcal C$ and~$\mathcal D$ consist of halves of the filled torus \emph{meridians}. Thus, our invariants make it possible to identify the meridians of a filled torus, although they are present in figure~\ref{razvertka} on equal grounds with the other lines.

\begin{remark}
Of course, by glueing two such triangulated filled tori in various ways, one can obtain manifolds $S^2\times S^1$ or~$S^3$, and if it is also permitted to pass to other triangulations, then all lens spaces as well. We leave the investigation of algebraic structures appearing this way for futher works.
\end{remark}

\subsection{Lens spaces without a tubular neighborhood of an unknot}

Recall that ``unknots'' in lens spaces~$L(p,q)$, studied in paper~\cite{dkm}, are the least knotted loops going along nontrivial elements of group~$\pi_1\bigl(L(p,q)\bigr)$. In the mentioned work, invariants were calculated for manifolds with boundary obtained by removing a tubular neiborhood of such unknot from~$L(p,q)$. In terms of the present paper, these are nothing but zero level invariants corresponding to various boundary triangulations, which were related in~\cite{dkm} to different \emph{framings} of the unknot.

We present here the results obtained using formulas of paper~\cite{dkm} for integer framings of unknots in~$L(7,1)$. The integrality of a framing means, in our terms, that the manifold boundary is triangulated as in figure~\ref{razvertka}, the meridian~$ADA$ (or~$CBC$) being fixed (in~\cite{dkm}, we also considered framings differing from these in half-revolution). Framings are enumerated by a number~$m\in \mathbb Z$; number~$n=1,2,3$ indicates how many times the unknot goes around the generator of the fundamental group.\footnote{For the remaining details --- to which framing corresponds $m=0$, in which direction $m$ grows and decreases, what specifically is chosen for the generator of the fundamental group --- the reader is referred to paper~\cite{dkm}.} Below in formulas (\ref{ktilde})--(\ref{n=3}) we present invariants obtained according to formula~(\ref{inv-b}), with slightly changed notations: the implied subscripts $(\mathcal C, \mathcal D)=(\emptyset,\emptyset)$ are omitted, in their place we put the number~$n$; number~$m$ stands for a superscript; these numbers, together with the name of a lens space, determine a manifold with boundary and the boundary triangulation.

We remark first that the dependence of the invariant on coordinates assigned to boundary vertices reduces to a simple multiplier:
\begin{equation}
I_n^{(m)}\bigl(L(p,q)\bigr) = (6V_{ABCD})^4 \tilde I_n^{(m)}\bigl(L(p,q)\bigr),
\label{ktilde}
\end{equation}
where the value with tilde no longer depends on these coordinates (this applies to all $p$ and~$q$). Then, we get from formulas of paper~\cite{dkm}:\footnote{Our $\tilde I_n^{(m)}\bigl(L(p,q)\bigr)$ corresponds to the value $1/I_n^{(2m)}\bigl(L(p,q)\bigr)$ in the notations of~\cite{dkm}.}
\begin{align}
\tilde I_1^{(m)}\bigl(L(7,1)\bigr) &= -7^2 (1-7m)^2 (6+7m)^2   ,  \label{n=1} \\
\tilde I_2^{(m)}\bigl(L(7,1)\bigr) &= -7^2 (3-7m)^2 (10+7m)^2  ,  \\
\tilde I_3^{(m)}\bigl(L(7,1)\bigr) &= -7^2 (-5-7m)^2 (12+7m)^2 .  \label{n=3}
\end{align}

The numbers obtained according to formulas (\ref{n=1})--(\ref{n=3}) for all possible~$m$, are \emph{all different}. Thus, for~$L(7,1)$ our zero level invariant can distinguish all unknots with all framings. Moreover, it never vanishes, that is, it behaves in a completely different way compared to the filled torus.

\section{Discussion}
\label{discussion}

In this paper, the first step has been done to constructing a topological field theory in the spirit of M.~Atiyah's axioms, based on torsions of acyclic complexes of geometric nature. The desire to build such a theory is motivated by the following reasons.

\emph{Finite dimensionality}: the theory is based on a finite manifold triangulation, it involves no functional integrals.

\emph{Any dimension of the manifold}: recall once more the works~\cite{33,24,15} where similar geometric torsions have been constructed for four-dimensional manifolds.

\emph{Richness and diversity of appearing theories}: in the present paper, we used algebraic complexes based on the Euclidean three-dimensional geometry. Meanwhile, our already made calculations~\cite{kkm} for a geometry related to the group $\mathrm{SL}(2,\mathbb C)$ give essentially new results.

\emph{Possibility to introduce $q$-commutation relations}: one of ``non-abelian'' modifications of our theory, proposed by Rinat Kashaev~\cite{kashaev-private}, admits a reduction which imposes Weyl relations $ba=qab$ on noncommuting variables.

As our nearest aim we see, however, the generalization of constructions presented in this paper for manifolds with any number of boundary components.

We conclude this discussion with the following remarks. As was mentioned in the end of section~\ref{skraem}, the components of invariant vector are not independent. Relations between them follow from Sylvester relations~\cite{sylvester,prasolov} among the minors of a matrix.

The simplest Sylvester relation has a three-term bilinear form and deals with the determinant of a matrix, a minor obtained by removing two rows and two columns, and four intermediate minors. It reminds, so, of the known ``free-fermion condition'' $a_+ a_- + b_+ b_- = c_+ c_-$ for the six-vertex model in statistical physics. There are also other indications at the presence of ``free fermions'' in our theory. To explain how it can happen that some of geometric complexes admit an introduction of $q$-commutation relations among some values, we recall that such situation has already been encountered in mathematical physics. Namely, some models in discrete quantum field theory are considered in paper~\cite{KKS} which, on the one hand, are free-fermion, and on the other hand, admit $q$-commutation relations. The matter is that when we have ``many enough'' free fermions, we can impose additional structures on them.

\subsection*{Acknowledgements}

The author thanks R.M.~Kashaev and E.V.~Martyushev for fruitful discussions. This work was performed with support from Russian Foundation for Basic Research, Grant 07-01-96005-r\_ural\_a.

\end{document}